\documentclass[10pt,reqno]{amsart}
\usepackage{jfong}
\usepackage{physics}
\usepackage{enumerate}
\newcommand{\del}[1]{\partial#1}
\newcommand{\delb}[1]{\overline\partial#1}
\newtheorem{hyp}[defn]{Hypothesis}
\title{On Chow stability and balanced embeddings}
\author{Ho Leung Fong}

\begin{document}
\maketitle

\begin{abstract}
An important result of Zhang states that for a projective variety, the existence of a balanced embedding is equivalent to Chow stability. In this paper, we shall prove that Chow stability implies that a balanced embedding exists via the continuity method. Our proof is conditional on a technical hypothesis about restrictions of Hamiltonians to subschemes of projective space.
\end{abstract}

\section{Introduction}
A projective variety has many embeddings into projective space. It is thus natural to ask: is there a ``best possible'' such embedding? Zhang gave a precise answer to this question in 1996, where ``best possible'' is interpreted through certain integrals over the projective variety. Briefly, we say the embedding $X\hookrightarrow \mathbb P^n$ is balanced if \[\int_X\frac{z_i\bar z_j}{|z|^2}\omega_{FS}^m=\lambda\,\id\] for some $\lambda\in \C$, where the $z_j$ are the homogeneous coordinates and $\omega_{FS}$ is the Fubini-study metric. Ultimately these embeddings do not always exist, and Zhang characterised their existence through an algebro-geometric notion, called Chow stability. Zhang's proof in \cite{zhang_1996} used the variational techniques. There are several new proofs in \cite{luo_1998},\cite{phong_sturm_2004}, and \cite{wang_2004}. The aim of this paper is to give a more geometric proof that should give some more insight into the problem.

The theorem we shall prove is the following:
\begin{thm}
	\label{thm:main}
	Let $X\subset \mathbb P^n$ be a smooth projective variety of dimension $m$. Assume $X$ is not contained in any hyperplane and has discrete automorphism group. If $X$ is Chow stable, then it admits a balanced embedding.
\end{thm}
Rather than variational techniques, we use a continuity method. Continuity methods are frequently used in studying the Yau-Tian-Donaldson conjecture, which relates K-stability to the existence of constant scalar curvature Kahler metrics, and which can be seen as an infinite dimensional analogue of our finite dimensional problem. There are many examples, such as \cite{chen_donaldson_sun_2014}, \cite{aoi2021uniform}, and \cite{chen_cheng_2021}. It is thus natural to expect that we can apply the continuity method to our problem. 

A key tool in our work will be to relate the vanishing of Hamiltonians on certain subschemes of projective space to their vanishing on projective space itself, which we make a hypothesis. This property is, for example, key part of the proof that the existence of balanced embeddings implies Chow stability. The hypothesis holds for smooth subvarieties by Lempert \cite{lempert2021bergman}.

\section*{Acknowledgements}
I would like to express my sincere gratitude to my supervisor Dr Ruadhaí Dervan for patiently answering my questions and suggesting the problem.
I would also like to thank Dr Yoshinori Hashimoto for helpful conversations.
This paper is funded by the Faculty of Mathematics in the University of Cambridge and a summer studentship associated to Ruadhaí Dervan's Royal Society University Research Fellowship.
\section{Preliminaries}
Let $X\subset\mathbb P^n$ be a projective variety of dimension $m$.
Define the Fubini-Study metric by $\omega_{FS}=i\del\delb \log|s|^2$, where $s$ is a local section of the projection map $\C^{n+1}\setminus \{0\}\to \mathbb P^n.$
\begin{defn*}[Balanced embedding]
	We say $X\hookrightarrow \mathbb P^n$ is balanced if \[\int_X\frac{z_i\bar z_j}{|z|^2}\omega_{FS}^m=\lambda\,\id\] for some $\lambda\in\C$. 
\end{defn*}
By taking the trace, we know that $\lambda$ must equal $\frac{1}{n+1}\int_X\omega_{FS}^{m}.$
Define \[F(g)_{ij}=i\int_{gX}\frac{z_i\bar z_j}{|z|^2}\omega_{FS}^m-\lambda\,\id\] for $g\in GL_{n+1}(\C)$, where $\lambda=\frac{i}{n+1}\int_X \omega_{FS}^m$.
We say that $X$ admits a balanced embedding if there exists a $g\in GL_{n+1}(\C)$ such that $F(g)=0$.

Note that $SU(n+1)$ acts on $\mathbb P^n$.
This induces a map $su(n+1)\to \text{Vect}(\mathbb P^n)$.
Denote the image of $\xi\in su(n+1)$ by $\underline \xi$.
Define 
\begin{equation}
	\label{eq:h}
h_{\xi}(z)=-i\sum_{j,k\ge 0}\frac{\xi_{kj}z_j\bar z_k}{|z|^2}
\end{equation}
for $\xi\in su(n+1)$ and $z\in \mathbb P^n$.
Then $h_\xi$ is a Hamiltonian for the action (c.f. \cite[Example 5.5]{szekelyhidi_2014}).
In other words, \[dh_{\xi}=-\iota_{\underline \xi}\omega_{FS}.\]

We want to investigate whether $X$ admits a balanced embedding.
It turns out that it is related to the notion of Chow stability.
Given a $\C^*$-action $\C^*\hookrightarrow SL_{n+1}(\C)$, there is an induced $\C^*$-action on the homogeneous coordinate ring $R=\C[x_0,\dots,x_n]/I_0$, where $I_0$ is the homogeneous ideal corresponding to the flat limit $X_0$.
Let $A_k$ be the generator of the $\C^*$-action on the degree $d$ part $R_k$ of $R$.
Then the total weight of the action is $w_k=\Tr(A_k)$.
For large $k$, $\Tr(A_k)$ is a degree $m+1$ polynomial. 
Denote the leading coefficient by $b_0$.
\begin{defn*}[Chow Stability]
	We call $X$ Chow stable if for all 1-parameter subgroup $\C^*\hookrightarrow SL_{n+1}(\C)$, \[b_0>0\text{ or }X_0=X.\]
\end{defn*}
This is not the usual definition of Chow stability, which is a notion of stability in the sense of geometric invariant theory, but is equivalent through the Hilbert-Mumford criterion \cite[Theorem 5.17]{szekelyhidi_2014}.

By \cite[Lemma 7.21]{szekelyhidi_2014}, \[b_0=\int_{X_0}h\frac{(\frac{1}{2\pi}\omega_{FS})^m}{m!}.\]
Thus, $X$ is Chow stable if and only if for all 1-parameter subgroup $\C^*\hookrightarrow SL_{n+1}(\C)$, \[\int_{X_0}h\omega_{FS}^m>0 \text{ or }X_0=X.\]

We shall assume the following hypothesis throughout the paper:
\begin{hyp}
	\label{hyp:hash}
	Define $h_\xi$ as in equation \eqref{eq:h}, and let $X \subset \mathbb P^n$ be a scheme.
	If $h_\xi$ is constant on $X$, then $\xi=0.$
\end{hyp}

Clearly some hypotheses are needed for this to be true, such as $X$ not being contained in any hyperplane, but as this is the condition we need, we simply make this a hypothesis. We will use this both for our smooth projective variety $X$ and also certain degenerations of $X$, which may be singular. 
As explained to us by Hashimoto, in the smooth case this hypothesis is equivalent to the injectivity of the Fubini-Study map. Indeed, by writing $i\xi$ as a difference of two strictly positive hermitian matrices $H_1 , H_2$, we get $h\xi =FS(H_1) - FS(H_2)$, where the notation follows \cite[Lemma 10, Equation (10)]{hashi2020}. For smooth projective varieties, it is proved in \cite{lempert2021bergman}. It will be important to use this also for singular subschemes of projective space.

\section{Proof of main result}

We shall prove theorem \eqref{thm:main} by the continuity method.
Let \[F_t(g)_{ij}=i\int_{gX}\frac{z_i\bar z_j}{|z|^2}\omega_{FS}^m+it\int_{gD}\frac{z_i\bar z_j}{|z|^2}-\lambda_t\,\id,\] where $\lambda_t=(1+t)\lambda$ and $D$ is a GIT stable finite point set contained in $X$. Such a $D$ exists by Lemma \ref{lem:stablept} if $X$ has $n+2$ points in general position.
Let \[I=\{t\in [0,\infty):\exists g\in GL_{n+1}(\C)\text{ such that }F_t(g)=0\}.\]
Our strategy is to first show that $I\cap (0,\infty)$ is open by the implicit function theorem.
Then we shall show that $I$ is non-empty by proving that for sufficiently large $t$, there exists $g_t$ such that $F_t(g_t)=0$.
Finally, we shall show that if $(t_j)_{j=1}^{\infty}$ is  decreasing sequence in $I$, then $\lim t_j\in I,$ assuming Chow stability of $X$.
This will allow us to conclude that $0\in I$, which implies that $X$ admits a balanced embedding.

Recall a set consisting of $n+2$ points in $\mathbb P^n$ is said to be in general position if any $n+1$ points are linearly independent.

\begin{lem}\label{lem:genpos}
	$X$ has a subset $D$ consisting of $n+2$ points in general position.
\end{lem}

\begin{proof}
	Assume this is false. Let $D=\{p_1,\dots,p_k\}$ be a subset of $X$ consisting of points in general position with the largest $k$.
	Then $1\le k \le n+1$.
	Let $H_i$ be the span of $D\setminus\{p_i\}$.
	Then \[X\subset\bigcup_{i=1}^k H_i.\] Indeed, if $p\in X\setminus \cup_{i=1}^k H_i$, then $\{x\}\cup D$ is a set consisting of points in general position. This contradicts the choice of $D$.

Since $X$ is irreducible, $X\subset H_i$ for some $i$.
Note that $H_i$ is a plane with dimension at most $n-1$.
This contradicts that $X$ is not contained in any hyperplane.
\end{proof}

We next show that such a set $D$ has the desired properties.
\begin{lem}
\label{lem:stablept}
Let $D$ be a set consisting of $n+2$ points in general position in $\mathbb P^n$. Then 
\begin{enumerate}[(i)]
	\item $D$ is a GIT stable point set;
	\item $D$ admits a balanced embedding.
\end{enumerate}
\end{lem}

\begin{proof}
	(i) By \cite[Proposition 7.27]{mukai_2003}, a collection of $n+2$ points in $\mathbb P^n$ is stable if and only if for all (proper) projective subspace $P\subset \mathbb P^n$, 
	\[\sharp P\cap D <\frac{n+2}{n+1}(\dim P+1)=\dim P+1+\frac{\dim P+1}{n+1}.\]
	For $\dim P<n$, this is equivalent to \[\sharp P\cap D \le \dim P+1.\]
	This is true if the $n+2$ points are in general position.

	(ii) Let $\zeta=e^{2\pi i/(n+2)}, v_b=[1:\zeta^b:\zeta^{2b}:\dots:\zeta^{nb}].$
	Then $E=\{v_0,v_1,\dots,v_{n+1}\}$ is a point set in $\mathbb P^n$ where any $n+1$ points are linearly independent by the properties of the Vandermonde matrix.

	By changing coordinates of $\mathbb P^n$, we can without loss of generality assume that $D=E$.
	We shall now show that $E$ is stable.
	Note that the $(a,b)$-th entry of the value of the moment map at $E$, where $0\le a \le n, 0\le b\le n+1$, is 
	\[
		\sum_{c=0}^{n+1}\frac{\zeta^{ac-bc}}{n+1}=
		\begin{cases}
			0 &\text{ if }a\neq b\\
			\frac{n+2}{n+1} &\text{ if }a=b.
		\end{cases}
	\]
	Thus, $E\hookrightarrow \mathbb P^n$ is a balanced embedding.
\end{proof}

From now on, we fix a subset $D$ of $X$ consisting of $n+2$ points in general position. This exists by Lemma \ref{lem:genpos} and has the properties in Lemma \ref{lem:stablept}. 

Let $J$ be the complex structure on $\mathbb P^n$.
We shall now prove a standard result regarding the Lie derivative of $\omega_{FS}.$

\begin{lem}\label{lem:lieomega}
	$\mathcal L_{\underline {iu}}\omega_{FS}=-2i\del\delb h_u$
\end{lem}

\begin{proof}
	Let $\underline u=\sum_{j\ge 1}(a_j \frac{\del}{\del w_j}+\bar a_j \frac{\del}{\del \bar w_j}), \omega_{FS}=\sum_{j,k\ge 1}ig_{jk}dw_j\wedge d\bar w_k.$
	Then $\underline{iu}=J\underline u=\sum_{j\ge 1}i(a_j \frac{\del}{\del w_j}-\bar a_j \frac{\del}{\del \bar w_j})$.
	Also, \[
		\iota_{\underline u}\omega_{FS}=\sum_{j,k\ge 1}ig_{jk}(a_jd\bar w_k-\bar a_k dw_j)=-dh_u
	\]
	so
	\begin{align*}
		\iota_{\underline {iu}}\omega_{FS}=\sum_{j,k\ge 1}g_{jk}(-a_jd\bar w_k-\bar a_k dw_j)j
		&=-J(\iota_{\underline u}\omega_{FS})\\
		&=J(dh_u)\\
		&=i(\del-\delb)h_u.
	\end{align*}

	By Cartan's magic formula, we have 
	\begin{align*}
		\mathcal L_{\underline {iu}}\omega_{FS}&=d\iota_{\underline {iu}}\omega_{FS}\\
		&=i(\del+\delb)(\del-\delb)h_u\\
		&=-2i\del\delb h_u.
	\end{align*}
\end{proof}

We shall now show that $I\cap (0,\infty)$ is open.
\begin{prop}
	The set $I\cap (0,\infty)$ is open.	
\end{prop}
\begin{proof}
	We shall apply the implicit function theorem.
	Assume $t\in I\cap (0,\infty)$ and $g\in GL_{n+1}(\C)$ satisfy $F_t(g)=0$.
	Define 
	\begin{align*}
		G\colon su(n+1)\times (0,\infty)&\to su(n+1)\\
		(u,t)&\mapsto F_t(e^{iu}g).
	\end{align*}
	By the implicit function theorem, to prove the proposition, it suffices to show that \[dG_{(0,t)}(\cdot,0):su(n+1)\to su(n+1)\] is invertible.
	Equivalently, we need to show that for all $u\in su(n+1)\setminus\{0\}$, 
	\begin{align*}
		0&\neq dG_{(0,t)}(u,0)\\
		&=\frac{d}{ds}\bigg |_{s=0}G(su,t)\\
		&=\frac{d}{ds}\bigg|_{s=0} F_t(e^{isu}g).
	\end{align*}

	Note that 
	\begin{align*}
		\frac{d}{ds}\bigg |_{s=0}F(e^{isu}g)_{ij}&=i\frac{d}{ds}\bigg |_{s=0}\int_{gX}(e^{isu})^*\left( \frac{z_i\bar z_j}{|z|^2}\omega_{FS}^m \right)\\
		&=i\int_{gX}\mathcal L_{\underline {iu}}\left( \frac{z_i\bar z_j}{|z|^2}\omega_{FS}^m \right)\\
		&=i\int_{gX} \left[ \iota_{\underline{iu}}d\left( \frac{z_i\bar z_j}{|z|^2} \right)\omega_{FS}^m-2mi\frac{z_i\bar z_j}{|z|^2}(\del\delb h_u)\wedge \omega_{FS}^{m-1} \right]
	\end{align*}
	by Cartan's magic formula and Lemma \ref{lem:lieomega}.
	
	By properties of moment maps, 
	\begin{align*}
		\sum_{i,j\ge 0}\frac{d}{ds}\bigg |_{s=0}F(e^{isu}g)_{ij}u_{ji}&=
		\int_{gX}\left[ \omega_{FS}(\underline u,\underline{iu})\omega_{FS}^m+2mih_u (\del\delb h_u)\wedge \omega_{FS}^{m-1} \right]\\
		&=\int_{gX}\left[ \omega_{FS}(\underline u,J\underline{u})\omega_{FS}^m-2mi\del h_u \wedge \delb h_u\wedge \omega_{FS}^{m-1} \right]\\
		&=\int_{gX}\left[ |\textrm{grad} h_u|^2\omega_{FS}^m-2|\del h_u|_{gX}|^2\omega_{FS}^m \right]\\
		&=\int_{gX}|(\textrm{grad} h_u)^{\perp}|^2\omega_{FS}^m
	\end{align*}
	because $\underline u=J\textrm{grad} h_u$ and $|\del h_u|_{gX}|^2=\frac{1}{2}|\textrm{grad} h_u|_{gX}^2$. Here $(\textrm{grad} h_u)^{\perp}$ is the component of $\textrm{grad} h_u$ orthogonal to $gX$.

	Similarly, we get 
	\[
	\sum_{i,j\ge 0}\frac{d}{ds}|_{s=0} F_t(e^{isu}g)u_{ji}=\int_{gX}|(\textrm{grad} h_u)^{\perp}|^2\omega_{FS}^m+\int_{gD}|\textrm{grad} h_u|^2.
	\]

	Suppose that $\frac{d}{ds}|_{s=0} F_t(e^{isu}g)=0.$
	Then  \[0=(\textrm{grad} h_u)^{\perp}=\textrm{grad} h_u|_{gD}.\]
	By assumption, $X$ has discrete automorphism group.
	Thus, $\textrm{grad} h_u=0$ on $X$, so $h_u$ is constant on $X$.
	By Hypothesis \ref{hyp:hash}, this implies $u=0,$ as desired.
\end{proof}

By a similar argument with the implicit function theorem, we can show that $I$ is non-empty.
\begin{prop}
	The set $I$ is non-empty.
\end{prop}

\begin{proof}
	By Lemma \ref{lem:stablept}, $D$ has a balanced embedding $g$.
	Consider the map \[f:(g,s)\mapsto si\int_{gX}\frac{z_i\bar z_j}{|z|^2}\omega_{FS}^m+i\int_{gD}\frac{z_i\bar z_j}{|z|^2}-\lambda_s \id.\]
	By assumption, $f(0)=0$.
	By the same argument using the implicit function theorem, we can show that for each $s$ in a neighbourhood of $0$, there is a $g_s\in GL_{n+1}\C$ such that $f(g_s,s)=0$.
	Pick any such $s>0$. Then $F_{1/s}(g_s)=0$, so $I\neq \varnothing.$
\end{proof}

We shall now prove the main theorem.
\begin{proof}[Proof of Theorem ~\ref{thm:main}]
	This is equivalent to saying that $0\in I$. By the previous propositions, it suffices to show that if $(t_j)$ is a decreasing sequence in $I$ converging to $t_{min}$, then $t_{min}\in I$.	

	By definition of $I$, for each $j$, there exists $g_j\in GL_{n+1}(\C)$ such that $F_{t_j}(g_j)=0$.
	By properness of the Hilbert scheme, we obtain limits $X_{min}=\lim_{j}g_jX$ and $D_{min}=\lim_{j}g_jD$ for some schemes $X_{min}$ and $D_{min}$.
	As $j\to\infty$, 
	\[F_{t_j}(g_j)\to i\int_{X_{min}}\frac{z_i\bar z_j}{|z|^2}\omega_{FS}^m+it_{min}\int_{D_{min}}\frac{z_i\bar z_j}{|z|^2}-\lambda_{t_{min}}\id,\] so

	\begin{equation}\label{eq:bal}
		i\int_{X_{min}}\frac{z_i\bar z_j}{|z|^2}\omega_{FS}^m+it_{min}\int_{D_{min}}\frac{z_i\bar z_j}{|z|^2}-\lambda_{t_{min}}\id=0.
	\end{equation}

	We claim that there is a one parameter subgroup $\rho:\C^*\hookrightarrow SL_{n+1}(\C)$ such that $\lim_{s\to 0}\rho(s)X=X_{min}$ and $\lim_{s\to 0}\rho(s)D=D_{min}.$
	By applying \cite[proposition 1]{donaldson2010stability} to the nested Hilbert scheme $N$ containing $(X,D)$, we just need to show that \[Aut(X_{min},D_{min})=\{g\in SL_{n+1}(\C):gX_{min}=X_{min}, gD_{min}=D_{min}\}\] is reductive.
	This was a key idea in \cite[Section 5]{chen_donaldson_sun_2014}.
	\begin{claim*}
		The automorphism group $Aut(X_{min},D_{min})$ is reductive.
	\end{claim*}
	There are multiple ways of proving this, all of which seem to rely on an application of Hypothesis \ref{hyp:hash} to $X_{min}.$ The clearest way of seeing this is the following.
	Firstly, we claim that the existence of the balanced embedding for $(X_{min},D_{min})$ implies Chow polystability of $(X_{min},D_{min})$. This relies on convexity arguments, and as in \cite[Lemma 7.19]{szekelyhidi_2014}, the key point is to use Hypothesis \ref{hyp:hash} to obtain strict convexity and hence Chow polystability. Then, viewing $(X_{min},D_{min})$ as a point in the nested Hilbert scheme, it is a standard fact that polystable points have reductive stabiliser, proving the claim \cite[Section 2]{MSMF_1973__33__81_0}.

	By changing coordinates, we can assume that $\rho(s)$ is diagonal with entries $t^{\lambda_0},t^{\lambda_1},\dots,t^{\lambda_n}$.
	By setting $i=j$ in equation \ref{eq:bal}, multiplying it by $\lambda_j$ and then summing over $j$, we get
	\begin{equation}\label{eq:ham}
		\sum_{j=0}^n i\int_{X_{min}}\frac{\lambda_j|z_j|^2}{|z|^2}\omega_{FS}^m+\sum_{j=0}^n it_{min}\int_{D_{min}}\frac{\lambda_j|z_j|^2}{|z|^2}=0.
	\end{equation}
	By the Chow stability of $X$ and $D$, equation \ref{eq:ham} implies that $X_{min}=X$ and $D_{min}=D$. By equation \ref{eq:bal}, this implies that $t_{min}\in I$, as desired.
\end{proof}

\begin{rmk}
	It is essential that we apply Hypothesis \ref{hyp:hash} to $X_{min}$ for our arguments.
\end{rmk}

\bibliography{research}
\bibliographystyle{apalike}

\end{document}